\documentclass[12pt]{amsart}

\usepackage{amssymb}
\usepackage{amscd}
\usepackage{hyperref}

\title[A far-flung Gorenstein numerical semigroup attaining the HKS bound]
{Existence of a far-flung Gorenstein numerical semigroup attaining the Herzog--Kumashiro--Stamate bound} 
\author{Akihiro Sugawara}

\address{Graduate School of Science and Engineering, Yamagata University, Kojirakawa-machi 1-4-12, Yamagata 990-8560, Japan} 
\email{s251436d@st.yamagata-u.ac.jp}

\newtheorem{theorem}{Theorem}
\newtheorem*{theorem*}{Theorem}

\newtheorem{fact}{Fact} 
\newtheorem*{fact*}{Fact} 
\newtheorem{lemma}{Lemma} 
\newtheorem{proposition}{Proposition}
\newtheorem{corollary}{Corollary} 

\newtheorem{question}{Question}
\newtheorem*{question*}{Question}

\theoremstyle{definition}
\newtheorem{definition}{Definition}
\newtheorem{example}{Example} 
\newtheorem{remark}{Remark}

\newtheorem*{Rohrbach}{The Rohrbach problem}

\def\e{{\mathrm e}}
\def\g{{\mathrm g}}

\def\l{{\mathrm{l}}}
\def\m{{\mathrm m}}
\def\n{{\mathrm n}}
\def\t{{\mathrm t}}

\def\F{{\mathrm F}}

\def\PF{{\mathrm {PF}}}
\def\rPF{{\mathrm {rPF}}}

\def\msg{{\mathrm{ msg }}}

\def\N{\mathbb{N}}

\def\Z{\mathbb{Z}}

\def\l{\mathrm{l}}

\def\int{\mathrm{int}}

\def\Ks{K(S)}
\def\Cs{C(S)}
\def\Ms{M(S)}

\def\tr{\mathrm{tr}}
\def\trK{\tr(\Ks)}
\def\res{\mathrm{res}}

\def\ress{\res(S)}

\def\ovn{\overline{n}}
\def\smz {\setminus\{0\}}

\begin{document}
\begin{abstract}
    We address a question posed by Herzog, Kumashiro, and Stamate 
    regarding the upper bound for the multiplicity of any far-flung Gorenstein numerical semigroup. 
    We answer this question in the affirmative 
    by presenting a method for constructing certain far-flung Gorenstein numerical semigroups with maximal reduced type. 
    Furthermore, using this method, we consider a question raised by Herzog, Hibi, and Stamate.
    More precisely, for any integer $t\ge5$, we construct a certain far-flung Gorenstein numerical semigroup with type $t$, which is a counterexample to the question.
\end{abstract}

\maketitle

\section*{Introduction}
The far-flung Gorenstein property was introduced by Herzog, Kumashiro, and Stamate  
to study one-dimensional Cohen--Macaulay rings whose
trace ideal of the canonical module is as small as possible (see \cite{zbMATH07648549}). 
Let $S$ be a numerical semigroup, and let $k[[S]]:=k[[t^s \mid s\in S]]$ 
be the numerical semigroup ring associated to $S$, where $k$ is a field. 
We say that $S$ is far-flung Gorenstein if $k[[S]]$ is far-flung Gorenstein. 

It is well-known that for any numerical semigroup $S \neq \N$,
the multiplicity $\m(S)$ of $S$ 
is at least $\t(S)+1$, where $\t(S)$ is the type of $S$. 
On the other hand, there is no upper bound for the multiplicity of a numerical semigroup $S \neq \N$ 
in terms of its type in general (for example, see \cite[Proposition 1]{sugawara2025classificationsymmetricnumericalsemigroups}). 
Surprisingly, it is shown in \cite{zbMATH07648549} that there exists
an upper bound for the multiplicity of any far-flung Gorenstein numerical semigroup,
which comes from the Rohrbach problem (see {\cite{zbMATH03023950}, \cite{zbMATH06223909}}) in additive number theory as follows:

\begin{fact}[{\cite[Corollary 5.3]{zbMATH07648549}}]\label{fact1}
    Let $S$ be a far-flung Gorenstein numerical semigroup. Then
    \[
\m(S) \le \ovn(\t(S)),
    \]
    where $\ovn(\t(S))$ is the solution of the Rohrbach problem for the type $\t(S)$ of $S$. 
\end{fact}

In the same paper, Herzog, Kumashiro, and Stamate raised the following question:

\begin{question}[{\cite[Question 5.5]{zbMATH07648549}}]\label{question1}
    Is it true that for any $t \ge 2$, there exists a far-flung Gorenstein numerical semigroup $S$ of type t with $\m(S)=\overline{n}(t)$?
\end{question}

This paper answers this question affirmatively by proving the following theorem.

\begin{theorem*}[{Theorem \ref{Main}}]
    For any integer $t\ge 2$, there exists a far-flung Gorenstein numerical semigroup
    such that $\t(S)=t$ and $\m(S)=\ovn(\t(S))$. 
\end{theorem*}

To prove this theorem, we first show that if $S$ is a far-flung Gorenstein numerical semigroup with $\m(S)=\ovn(\t(S))$, then 
$S$ has maximal reduced type. 
From this observation, we focus on providing a method of constructing far-flung Gorenstein numerical semigroups with maximal reduced type. 
Furthermore, our method produces counterexamples to \cite[Question 2.3]{zbMATH07395083} (see Question \ref{comparison}). 
It is known that this statement holds true when $\t(S)\le3$ (see \cite[Theorem 2.4]{zbMATH07574809}), and that  
there exists a counterexample with $\t(S)=5$ (see \cite[Example 2.5]{zbMATH07574809}). 
We show that there exists a counterexample with type $t$ for any integer $t\ge5$. 
In addition, we present a counterexample to \cite[Question 2.6]{zbMATH07574809}, which corresponds to the case $\t(S)=4$ in \cite[Question 2.3]{zbMATH07395083}. 

The structure of the article is as follows. 
In Section \ref{Preliminaries}, we fix the notation, and we recall some definitions and results about numerical semigroups. 
In Section \ref{FFGwithMRT}, we recall the Rohrbach problem, 
and we provide a characterization of far-flung Gorenstein numerical semigroups 
by means of $\rPF(S)$ (see Lemma \ref{lemma1}). 
In addition, we show that if $S$ is a far-flung Gorenstein numerical semigroup with $\m(S)=\ovn(\t(S))$, 
then $S$ has maximal reduced type (see Proposition \ref{proposition1}).
In Section \ref{MainResults}, 
we prove Theorems \ref{theorem1} and \ref{Main} to answer Question \ref{question1}.  
Finally in Section \ref{res(S)}, applying the same method as in Section \ref{MainResults}, 
we provide counterexamples to \cite[Question 2.3]{zbMATH07395083} (see Proposition \ref{example_Rohrbach} and Theorem \ref{theorem_example}).
In addition, we also introduce a counterexample to \cite[Question 2.6]{zbMATH07574809} (see Example \ref{type4}).

%%%%%%%%%%%%%%%%%%%%%%%%%%%%%%%%%%%%%%%%%%%%%%%%%%%%%%%%%%%%%%%

\section{Preliminaries}\label{Preliminaries}
First, we fix the notation and recall some definitions and basic facts on numerical semigroups. 
Let $\Z$ and $\N$ denote the set of integers and the set of non-negative integers, respectively. 
The cardinality of the set $A$ is denoted by $|A|$.   
For any integers $x\le y$, 
the set of integers $\{z \in \Z\mid x \le z \le y\}$ is represented by $[x,y]$. 
We define $\langle n_1,\dots, n_e \rangle =\{\sum_{i=1}^{e} \lambda_i n_i \mid \lambda_i \in \N\}$, where $n_1,\dots,n_e$ are non-negative integers. 
A numerical semigroup $S$ is a submonoid of $(\N,+)$ such that $\N\setminus S$ is finite. 
Let $S$ be a numerical semigroup. 
The smallest element of $S \setminus \{0\}$ is called the multiplicity of $S$, and it is denoted by $\m(S)$. 
The set $\N \setminus S$ is called the set of gaps of $S$,
and its cardinality is said to be the genus $\g(S)$ of $S$. 
It is well-known that every numerical semigroup has a unique minimal system of generators, and it is denoted by $\msg(S)$. 
The cardinality $\e(S)$ of $\msg(S)$ is called the embedding dimension of $S$. 
Note that the inequality $\e(S)\le\m(S)$ holds. We say that $S$ has maximal embedding dimension (MED-semigroup) if $\e(S)=\m(S)$.

The largest integer in $\Z \setminus S$ is called the Frobenius number of $S$, and it is denoted by $\F(S)$. 
Let $\PF(S):=\{ x \in \Z \setminus  S \mid x+s \in S \text{~for all~} 0 \neq s \in S \}$ be the set of 
all pseudo-Frobenius numbers of $S$. 
The cardinality $\t(S)$ of $\PF(S)$ is said to be the type of $S$. 
We say that $S$ is symmetric 
if $x \in \Z \setminus S$ implies that $\F(S)-x \in S$, or 
equivalently, if $\t(S)=1$. 
We also define
\[
    \rPF(S) = [\F(S)-\m(S)+1,\F(S)]\setminus S. 
\]
The cardinality of $\rPF(S)$ is called the reduced type $s(S)$ of $S$ (see \cite{zbMATH07423434}, \cite{arXiv:2306.17069}, \cite{sugawara2025classificationsymmetricnumericalsemigroups}).
Note that $\F(S) \in \rPF(S) \subset \PF(S)$. 
Thus, we obtain  
\[
1\le s(S)\le \t(S).
\] 
Then, we say that $S$ has maximal reduced type if $s(S)=\t(S)$, i.e., $\rPF(S)=\PF(S)$. 

Given two sets $A$ and $B$, we define $A+B=\{a+b\mid a\in A,b\in B\}$.  
A relative ideal $I$ of $S$ is a set $I \subset \Z$ such that 
$I+S \subset I$ and $x+I \subset S$ for some $x \in S$, where $x+I:=\{x+a \mid a \in I\}$.  
The following three relative ideals are very important:
the maximal ideal $M(S):=S\setminus\{0\}$ of $S$, 
the conductor ideal $\Cs:=\{z \in \Z\mid z\ge \F(S)+1\}$ of $S$, 
and the standard canonical ideal $K(S):=\{x \in \Z \mid \F(S)-x\notin S\}$ of $S$.
Note that if $I$ and $J$ are relative ideals of $S$, 
then $I+J$ and $I-J:=\{z\in \Z \mid z+y \in I \text{~for all~} y \in J\}$
are relative ideals of $S$. 

The trace ideal of $\Ks$ is defined as $\trK:=\Ks+(S-\Ks)$. 
In general, the following fact is known:
\begin{fact}[{\cite[Proposition 2.1]{zbMATH07395083}}]\label{fact2}
\[
\Cs  \subset \trK \subset S. 
\]
\end{fact}
We recall that $S$ is symmetric if and only if $\Ks=S$, i.e., $0\in \trK$. 
Then, it follows that $\trK=S$ if and only if $S$ is symmetric. 
Hence, it is natural to consider the opposite case. 
\begin{definition}[\cite{zbMATH07648549}]
    A numerical semigroup $S$ is said to be far-flung Gorenstein if $\trK=\Cs$.
\end{definition}
The residue of $S$ is defined as $\ress:=|S\setminus\trK|$, 
and the cardinality of $S \setminus \Cs$ is denoted by $\n(S)$. Clearly, 
\[
\ress \le \n(S). 
\]
Then, $S$ is far-flung Gorenstein if and only if $\ress=\n(S)$ (see \cite[Corollary 2.2]{zbMATH07395083}). 
A numerical semigroup $S$ is called a half-line if $S=\Delta(m)$ for some integer $m\ge1$, where $\Delta(m):=\{x\in \N\mid x\ge m\}\cup \{0\}$. 
It is easy to see that $S$ is a half-line if and only if $\m(S)\ge \F(S)+1$, i.e., $\Ms \subset \Cs$. 
Moreover, if $S$ is a half-line and not symmetric, then $S$ is far-flung Gorenstein. 
Note that $S$ is symmetric and far-flung Gorenstein if and only if $S=\Cs$, i.e., $S=\N$. 
Therefore, it is enough to consider far-flung Gorenstein numerical semigroups with $\t(S)\ge2$.

%%%%%%%%%%%%%%%%%%%%%%%%%%%%%%%%%%%%%%%%%%%%%%%%%%%%%%%%%%%%%%%

\section{Far-flung Gorenstein numerical semigroups with maximal reduced type}\label{FFGwithMRT}

\begin{Rohrbach}[{\cite{zbMATH03023950}, \cite{zbMATH06223909}}]
Let $A$ be a finite set of non-negative integers, 
and let $n(A)$ be the integer such that $0,1,\dots,n(A)-1 \in A+A$ and $n(A) \notin A+A$. 
If $0 \notin A$, then let $n(A)=-1$. For a positive integer $t$, 
the Rohrbach problem asks to find the integer 
\[
\ovn(t):=\max\{n(A)\mid A \subset \N,|A|=t\}.
\]
\end{Rohrbach}

\begin{definition}[{\cite[Definition 7]{zbMATH06272466}}]
    We say that a finite set $A$ of non-negative integers is extremal if $n(A)=\ovn(|A|)$.  
\end{definition}

\begin{example} 
     \begin{itemize}
        \item[(1)] It is easy to see that $\{0\}$ and $\{0,1\}$ are extremal, thus $\ovn(1)=1$ and $\ovn(2)=3$. 
        \item[(2)] Consider $\{0,1,2\}$ and $\{0,1,3\}$. We see that they are extremal and $\ovn(3)=5$. 
        \item[(3)] Let $A \subset \N$ be an extremal finite set with $|A|=4$. Clearly, $0,1 \in A$. 
Since $3 \notin \{0,1\}+\{0,1\}$, it holds either $2 \in A$ or $\in 3\in A$. 
Hence, we observe that $A=\{0,1,3,4\}$ is extremal, and we have $\ovn(4)=9$.  
        \end{itemize}
\end{example}

The solution of the Rohrbach problem is known for $t \le 25$ (see \cite{zbMATH06272466}, \cite{zbMATH06223909}). 
For example, $\ovn(5)=13, \ovn(6)=17, \ovn(7)=21$, and $\ovn(8)=27$. 

\begin{fact}[{\cite[Proposition 6.1]{zbMATH07648549}}]\label{fact3}
    Let $S$ be a numerical semigroup. Then the following statements are equivalent:
    \begin{itemize}
        \item[(1)] $S$ is far-flung Gorenstein;
        \item[(2)] $\{0,\dots,\m(S)-1\}\subset \{(\F(S)-f)+(\F(S)-g) \mid f,g \in \PF(S)\}$.
        \end{itemize}
\end{fact}
Fact \ref{fact3} implies that Fact \ref{fact1}. 
From these facts, the classification of far-flung Gorenstein numerical semigroups 
with small type is given (see \cite{zbMATH07648549}, \cite{grigorescu2025farflunggorensteinnumericalsemigroup}).
Recall that $\rPF(S) \subset \PF(S)$. Using $\rPF(S)$, 
we give a characterization of far-flung Gorenstein numerical semigroups. 
\begin{lemma}\label{lemma1}
    Let $S$ be a numerical semigroup. Then the following statements are equivalent:
    \begin{itemize}
        \item[(1)] $S$ is far-flung Gorenstein;
        \item[(2)] $\{0,\dots,\m(S)-1\}\subset \{(\F(S)-f)+(\F(S)-g) \mid f,g \in \rPF(S)\}$;
        \item[(3)] $\m(S) \le n(B)$, where $B=\{\F(S)-f \mid f \in \rPF(S)\}$.
    \end{itemize}
\end{lemma}
\begin{proof}
    $(2)\Rightarrow(1)$. Since $\rPF(S)\subset \PF(S)$, it follows from Fact \ref{fact3} that $S$ is far-flung Gorenstein. 

    $(1)\Rightarrow(2)$. Let $x \in \{0,\dots,\m(S)-1\}$. By (1) and Fact \ref{fact3}, 
    we see that $x=(\F(S)-f)+(\F(S)-g)$ for some $f,g \in \PF(S)$. 
    We show that $f,g \in \rPF(S)$. Assume by contradiction that $f \notin \rPF(S)$, namely $f\le \F(S)-\m(S)$. 
    Then $x=(\F(S)-f)+(\F(S)-g)\ge \F(S)-f\ge \m(S)$, which is a contradiction. Thus, $f \in \rPF(S)$. 
    Similarly, we have $g\in \rPF(S)$. Therefore, $f,g\in \rPF(S)$. 

    $(2)\Leftrightarrow(3)$. It follows from the definition of $n(B)$.
\end{proof}

Recall that $1\le \t(S) \le \m(S)-1$ when $S \neq \N$ (see \cite[Corollary 2]{zbMATH03993709}). 
Then, $S$ is a MED-semigroup if and only if $\t(S)=\m(S)-1$ (see \cite[Corollary 3.2]{zbMATH05623301}).

\begin{corollary}\label{corollary1}
    Let $S \neq \N$ be a far-flung Gorenstein numerical semigroup. 
    Then 
    \[
    \t(S)+1\le \m(S) \le \ovn(s(S)).
    \]
\end{corollary}

\begin{lemma}\label{lemma2}
    Let $t \ge 1$ be an integer. Then $\ovn(t)<\ovn(t+1)$. 
\end{lemma}

\begin{proof}
    There exists an extremal finite set $A \subset \N$ with $|A|=t$. 
    By definition, we have $0,\dots,\ovn(t)-1 \in A+A$ and $\ovn(t) \notin A+A$. 
    Set $B=A \cup \{\ovn(t)\}$. Note that $0 \in A$ since $A$ is extremal. 
    Then $0,\dots,\ovn(t) \in B+B$. 
    Hence, we obtain $\ovn(t)+1\le n(B)\le \ovn(t+1)$ since $|B|=t+1$. 
    Therefore, $\ovn(t)<\ovn(t+1)$.
\end{proof}

Proposition \ref{proposition1} suggests that we need to construct a far-flung Gorenstein numerical semigroup with maximal reduced type 
to answer Question \ref{question1}. 
\begin{proposition}\label{proposition1}
    Let $S$ be a far-flung Gorenstein numerical semigroup with $\t(S)\ge2$.
    If $\m(S) \ge \ovn(\t(S)-1)+1$, then $S$ has maximal reduced type. 
    In particular, if $\m(S)=\ovn(\t(S))$, then $S$ has maximal reduced type. 
\end{proposition}
\begin{proof}
    By Corollary \ref{corollary1}, we have $\m(S)\le \ovn(s(S))$. 
    By assumption, $\ovn(\t(S)-1)+1\le\ovn(s(S))$. 
    According to Lemma \ref{lemma2}, 
    we obtain $\t(S)-1<s(S)$. 
    Since $s(S)\le \t(S)$, it follows that $s(S)=\t(S)$.
\end{proof}
The following example shows that the converse of Proposition \ref{proposition1} does not hold in general. 
\begin{example}
    Let $S=\langle 9, 10, 13, 15, 21 \rangle$. 
    Then, $\m(S)=9$ and $\rPF(S)=\PF(S)=\{ 11, 12, 14, 16, 17\}$.  
    Thus, we obtain $s(S)=\t(S)=5$ and $n(B)=13$, where $B=\{\F(S)-f\mid f \in \rPF(S)\}=\{ 0, 1, 3, 5, 6\}$. 
    It follows from Lemma \ref{lemma1} that $S$ is far-flung Gorenstein. 
    Since $\ovn(4)=9$, we have $\m(S)=9<10=\ovn(\t(S)-1)+1$. 
\end{example}

Proposition \ref{proposition1} is inspired by and extends \cite[Theorem 3.26]{arXiv:2306.17069}.
This generalization provides elementary proofs of \cite[Theorem 3.23 and 3.25]{arXiv:2306.17069}.

\begin{corollary}[{\cite[Theorem 3.23 and 3.25]{arXiv:2306.17069}}]\label{corollary2}
    Let $S \neq \N$ be a far-flung Gorenstein numerical semigroup. 
    Then the following hold:
    \begin{itemize}
        \item[(1)] If $\t(S)\le3$, then $S$ has maximal reduced type;
        \item[(2)] If $\t(S)=4$ and $S$ is not a MED-semigroup, then $S$ has maximal reduced type.
    \end{itemize}
\end{corollary}

\begin{proof}
    (1) By \cite[Proposition 3.22]{arXiv:2306.17069}, we obtain $s(S)\ge 2$. 
    Thus, $2\le s(S)\le \t(S)$. This implies that if $\t(S)=2$, then $s(S)=\t(S)$. 
    Then we may assume that $\t(S)=3$. 
    In this case, we see that $\m(S)\ge \t(S)+1=4=\ovn(\t(S)-1)+1$ since $\ovn(2)=3$. 
    Therefore, we obtain $s(S)=\t(S)$ by Proposition \ref{proposition1}. 
    
    (2) Since $S$ is not a MED-semigroup, we have $\t(S)+1<\m(S)$. 
    Hence, $\m(S)\ge\t(S)+2=6=\ovn(\t(S)-1)+1$ since $\ovn(3)=5$. 
    Similarly, we obtain $s(S)=\t(S)$.
\end{proof}

The following example shows that there exists a far-flung Gorenstein numerical semigroup $S$ with $\t(S)=5$
such that $S$ is not a MED-semigroup and $S$ does not have maximal reduced type. 
\begin{example}
    Let $S=\langle 9,15,28,29,40,41 \rangle$. 
    In this case, we have $\m(S)=9, \PF(S)=\{21, 31, 32, 34, 35\}$, and $\rPF(S)=\{31, 32, 34, 35 \}$. 
    Thus, $4=s(S)<\t(S)=5$ and $\t(S)+1<\m(S)$. Moreover, we see that     
    $B:=\{\F(S)-f\mid f\in \rPF(S)\}=\{ 0, 1, 3, 4 \}$ and $n(B)=9$. 
    According to Lemma \ref{lemma1}, $S$ is far-flung Gorenstein.     
\end{example}

%%%%%%%%%%%%%%%%%%%%%%%%%%%%%%%%%%%%%%%%%%%%%%%%%%%%%%%%%%%%%%%

\section{Main results}\label{MainResults}

Let $m \ge 2$ be an integer, and let $A$ be a finite set of non-negative integers such that $0\in A \subset [0,m-2]$.
It is easy to see that 
\[
S(m,A):=\Delta(m) \setminus \{(2m-1)-a \    \mid a\in A\}
\]
is a numerical semigroup, where $\Delta(m)=\{x \in \N \mid x\ge m\} \cup \{0\}$. 
Let $\l(S)$ denote the cardinality of $\Ks \setminus S$. 
\begin{proposition}\label{construction}
    Let $m \ge 2$ be an integer, and let $A$ be a finite set of non-negative integers such that $0\in A \subset [0,m-2]$. 
    Then the following hold:
    \begin{itemize}
        \item[(1)] $\m(S(m,A))=m$, $\F(S(m,A))=2m-1$;
        \item[(2)] $\rPF(S(m,A))=\{(2m-1)-a \mid a\in A\}$, $s(S(m,A))=|A|$. 
        In particular, $\{\F(S(m,A))-f\mid f \in \rPF(S(m,A))\}=A$;
        \item[(3)] $S(m,A)$ is far-flung Gorenstein if and only if $m\le 
        n(A)$;
        \item [(4)] $\n(S(m,A))= m-|A|+1$, $\l(S(m,A))=2|A|-2$; 
        \item [(5)] $\n(S(m,A))>\l(S(m,A))$ if and only if $m>3(|A|-1)$. 
    \end{itemize}
\end{proposition}

\begin{proof}
    (1), (2) Since $0\in A \subset [0,m-2]$, it follows that 
    \[
2m-1 \in \{(2m-1)-a \mid a \in A\} \subset [m+1,2m-1]. 
    \]
Thus, $\m(S(m,A))=m$ and $\F(S(m,A))=2m-1$. 
Furthermore, we obtain $\rPF(S(m,A))=[m,2m-1]\setminus S=\{(2m-1)-a \mid a\in A\}$.

(3) The assertion follows from (1), (2), and Lemma \ref{lemma1}. 

(4) By the definition of $S(m,A)$, we see that $\n(S(m,A))=m-|A|+1$. 
Recall that for a numerical semigroup $S$, $\g(S)=\n(S)+\l(S)$. 
Since $\g(S(m,A))=m-1+|A|$, we have $\l(S(m,A))=2|A|-2$. 

(5) The assertion follows from (4).
\end{proof}

\begin{example}
    \begin{itemize}
        \item [(1)] Let $S=S(5,\{0,3\})$. Then, $S=\Delta(5)\setminus \{6,9\}=\langle 5, 7, 8, 11 \rangle$. 
        From Proposition \ref{construction}, we obtain $\m(S)=5, \F(S)=9$, and $\rPF(S)=\{6,9\}$. 
        Note that $n(\{0,3\})=1$. Thus, $S$ is not far-flung Gorenstein by Proposition \ref{construction}. 
        \item [(2)] Let $S=S(5,\{0,1,2\})$. Then, $S=\Delta(5)\setminus\{7,8,9\}=\langle 5,6,13,14 \rangle$. 
        From Proposition \ref{construction}, we obtain $\m(S)=5, \F(S)=9$, and $\rPF(S)=\{7,8,9\}$. 
        Note that $n(\{0,1,2\})=5$. Thus, $S$ is far-flung Gorenstein by Proposition \ref{construction}. 
    \end{itemize}
\end{example}

\begin{theorem}\label{theorem1}
    Let $S$ be a numerical semigroup such that $\m(S)<\F(S)+1$, 
    and let $B=\{\F(S)-f \mid f \in \rPF(S)\}$.
    Assume that $\max B+1\le n(B)$ and $\max B+n(B)<2\m(S)$.
    Then the following hold: 
    \begin{itemize}
    \item[(1)] If $x \in [1,\m(S)-1] \setminus B$, then $x \notin \PF(S)$;
    \item[(2)] If $x \in B \smz$, then $x \notin \PF(S)$.
    \end{itemize}
    In particular, if $x \in [1,\m(S)-1]$, then $x \notin \PF(S)$. 
\end{theorem}

\begin{proof}
    (1) Let $x \in [1,\m(S)-1]\setminus B$. 
    We show that $\F(S)-x \in S$. 
    By assumption, we see that $\F(S)-x \in [\F(S)-\m(S)+1,\F(S)-1]$. 
    If $\F(S)-x \notin S$, then $\F(S)-x \in \rPF(S)$. 
    This implies that $x \in B$, which is a contradiction. 
    Thus, $\F(S)-x \in S$.
    Since $x\le \m(S)-1<\F(S)$, we have $\F(S)-x \in S \smz$.  
    It follows that $x \notin \PF(S)$. 
    
    (2) Let $x \in B\setminus\{0\}$. Note that $n(B) \notin B+B$. 
    This implies that $n(B)-x \notin B$. Since $\max B +1\le n(B)$ and $x>0$, we have $1\le n(B)-x<n(B)$. 
    Thus, $n(B)-x  \in(B+B)\setminus B$. 
    It follows that 
    \[
    n(B)-x=a+b
    \]
    for some $a,b\in B$ with $0<a,b<n(B)-x$.
    We now show that $n(B)-a<\m(S)$ or $n(B)-b<\m(S)$. 
    Assume by contradiction that $n(B)-a\ge\m(S)$ and $n(B)-b\ge\m(S)$. 
    Then $2\m(S)\le (n(B)-a)+(n(B)-b)=2n(B)-(a+b)=n(B)+x$. 
    This is a contradiction since $\max B+n(B)<2\m(S)$. 
    Hence, we may assume that $n(B)-b <\m(S)$. 
    Since $b \in B$, we have $n(B)-b \notin B$ and $n(B)-b \in [1,\m(S)-1]$. 
    Using the same argument in (1), we obtain $\F(S)-(n(B)-b) \in S\smz$.
    Then, 
    \[
(\F(S)-a)-x =\F(S)-(a+x)=\F(S)-(n(B)-b) \in S \smz. 
    \]
    Note that $\F(S)-a \in \rPF(S)$ since $a \in B$. 
    Therefore, $x \notin \PF(S)$. 
\end{proof}

In the setting of Theorem \ref{theorem1}, if $S$ is far-flung Gorenstein, 
then $\max B+1\le \m(S) \le n(B)$ by Lemma \ref{lemma1}. 
Hence, we have the following corollary:
\begin{corollary}
    Let $S$ be a far-flung Gorenstein numerical semigroup such that $\m(S)<\F(S)+1$, 
    and let $B=\{\F(S)-f \mid f \in \rPF(S)\}$.
    Assume that $\max B+n(B)<2\m(S)$. 
    Then $x \notin \PF(S)$ for all $x \in [1,\m(S)-1]$. 
\end{corollary}

\begin{corollary}\label{corollary3}
    Let $m \ge 2$ be an integer, and let $A$ be a finite set of non-negative integers such that $0\in A \subset [0,m-2]$. 
    Set $S=S(m,A)$. 
    Then the following hold:
    \begin{itemize}
        \item [(1)] If $\max A+1\le n(A)$ and $\max A+n(A)<2m$, then $s(S)=\t(S)=|A|$. 
        \item [(2)] If $m \le n(A)$ and $\max A+n(A)<2m$, then $S$ is a far-flung Gorenstein numerical semigroup with $s(S)=\t(S)=|A|$;
        \item [(3)] If $m=n(A)$, then $S$ is a far-flung Gorenstein numerical semigroup with $s(S)=\t(S)=|A|$.
    \end{itemize}
\end{corollary}
\begin{proof} 
    (1) By Proposition \ref{construction}, we have $\m(S)=m$, $\F(S)=2m-1$, and
    \[
\{\F(S)-f\mid f \in \rPF(S)\}=A.
    \]
    Thus, $\rPF(S)=[\F(S)-\m(S)+1, \F(S)] \setminus S=[m,2m-1] \setminus S$. 
    We show that $\rPF(S)=\PF(S)$. Clearly, $\rPF(S)\subset \PF(S)$.
    Assume by contradiction that there exists $f \in \PF(S)$ such that $f \notin \rPF(S)$. 
    We see that $f \in [1, m-1]$. Note that $\F(S)>\m(S)$ since $m\ge 2$.
    Then it follows from Theorem \ref{theorem1} that $f \notin \PF(S)$. 
    This is a contradiction. Therefore, we have $\rPF(S)=\PF(S)$, i.e., $s(S)=\t(S)=|A|$.

    (2), (3) The assertions follow from (1) and Proposition \ref{construction}. 
\end{proof}

The following lemma is known, 
but we include a proof for the convenience of readers.
\begin{lemma}\label{lem_extremal}
    Let $t\ge 2$ be an integer, and let $A \subset \N$ be an extremal finite set with $|A|=t$. 
    Then the following hold:
    \begin{itemize}
        \item[(1)] $0,1\in A$;
        \item[(2)] $\max A+2\le n(A)$.
    \end{itemize}
\end{lemma}
\begin{proof}
    (1) is obvious. 
    (2) Since $1, \max A \in A$, we have $\max A+1 \in A+A$. 
    Thus, it is enough to show that $0,\dots,\max A \in A+A$, 
    i.e., $\max A +1\le n(A)$. 
    Assume by contradiction that $n(A)\le \max A$. 
    This implies that $\{0,\dots,n(A)-1\}\subset (A \setminus \{\max A\})+(A \setminus \{\max A\})$. 
    Set $B=(A \setminus\{\max A\}) \cup \{n(A)\}$. 
    Then, we obtain $n(A)+1\le n(B)$, which is a contradiction since $A$ is extremal. 
\end{proof}

\begin{theorem}\label{Main}
    Let $t\ge 2$ be an integer, and let $A \subset \N$ be an extremal finite set with $|A|=t$. 
    Set $S=S(n(A),A)$. Then the following hold:
    \begin{itemize}
        \item[(1)] $S$ is far-flung Gorenstein;
        \item[(2)] $s(S)=\t(S)=t$;
        \item[(3)] $\m(S)=\ovn(\t(S))$.
    \end{itemize}
\end{theorem}
\begin{proof}
    By Lemma \ref{lem_extremal}, we have $n(A)\ge 2$ and $0\in A \subset [0,n(A)-2]$. 
    Hence, it follows from Corollary \ref{corollary3} (3) 
    that $S$ is a far-flung Gorenstein numerical semigroup with $s(S)=\t(S)=|A|=t$. 
    Therefore, $\m(S)=n(A)=\ovn(t)=\ovn(\t(S))$ by Proposition \ref{construction}.     
\end{proof}

%%%%%%%%%%%%%%%%%%%%%%%%%%%%%%%%%%%%%%%%%%%%%%%%%%%%%%%%%%%%%%%

\section{The case where $\ress>\l(S)$}\label{res(S)}

Recall that $\ress=|S\setminus \trK|$, $\l(S)=|\Ks \setminus S|$, and $S$ is far-flung Gorenstein if and only if $\ress=\n(S)$. 
Note that $\g(S)=\l(S)+\n(S)$. Then, the following question is known:
\begin{question}[{\cite[Question 2.3]{zbMATH07395083}, \cite[Question 1.1]{zbMATH07574809}}]\label{comparison}
    Given a numerical semigroup $S$, is it true that 
    \[
    \ress\le \l(S)?
    \]
\end{question}

The statement holds for $\t(S)\le 3$ (see \cite[Theorem 2.4]{zbMATH07574809}). 
However, there exists a counterexample for Question \ref{comparison} when $\t(S)=5$ (see \cite[Example 2.5]{zbMATH07574809}). 

\begin{remark}
    As mentioned in \cite[Example 2.5 and 4.4]{zbMATH07574809}, 
\[ 
S=\langle 13,14,15,16,17,18,21,23 \rangle
\]
is a far-flung Gorenstein numerical semigroup with $\t(S)=5$ and $\m(S)=13$ such that $\ress> \l(S)$.  
On the other hand, it is easy to check that $S=S(m,A)$, where $m=13$ and $A=\{0,1,3,5,6\}$. 
In this case, since $A$ is extremal, it follows from Theorem \ref{Main} that 
$S$ is a far-flung Gorenstein numerical semigroup with maximal reduced type. 
\end{remark}
In this section, we show that for any integer $t\ge 5$, there exists a numerical semigroup with $\t(S)=t$ such that $\ress>\l(S)$.

\begin{proposition}\label{example_Rohrbach}
    Let $t\ge4$ be an integer, and let $s=\lfloor \frac{t}{2} \rfloor$. 
    Set 
    \[
    A=
\begin{cases}
            \{0,1,\dots,s,2s,\dots,s^2\}                            & \quad(t \text{ is even})    ; \\
            \{0,1,\dots,s,2s,\dots,s^2,(s+1)s\} & \quad(t \text{ is odd}).
        \end{cases}
    \]
    Put $S=S(n(A),A)$.
    Then the following hold:
    \begin{itemize}
        \item [(1)] $|A|=t$;
        \item [(2)] 
        \[ n(A)=\begin{cases}
            s^2+s+1=\frac{t^2}{4} +\frac{t}{2}+1           & \quad(t \text{ is even}); \\
            s^2+2s+1=\frac{t^2}{4}+\frac{t}{2}+\frac{1}{4} & \quad(t \text{ is odd}). 
        \end{cases}
    \]
    In particular, $\max A+2\le n(A)$; 
    \item [(3)] If $t\ge 9$,  then $3(t-1)<n(A)$;  
    \item [(4)] $S$ is a far-flung Gorenstein numerical semigroup with $s(S)=\t(S)=t$. 
Moreover, if $t\ge 9$, then $\res(S)>\l(S)$. 
    \end{itemize}
\end{proposition}

\begin{proof}
    (1) If $t$ is even, then $t=2s$, otherwise, $t=2s+1$. 
    This implies that $|A|=t$. 

    (2) First, we consider the case where $t$ is even. Then, we have $[0,s^2+s] \subset A+A$. 
    We show that $s^2+s+1 \notin A+A$. 
    Assume by contradiction that $s^2+s+1 \in A+A$, which implies that $s^2+s+1=a+b$ for some $a,b \in A$. 
    It is easy to see that $s\le a,b$. 
    Then, $s$ divides $s^2+s+1$, which is a contradiction. 
    Thus, $s^2+s+1 \notin A+A$. Hence, we have $n(A)=s^2+s+1$. 
    Since $t=2s$, we obtain $n(A)=\frac{t^2}{4} +\frac{t}{2}+1$. 
    We can use the same argument when $t$ is odd. 
    
    (3) The assertion follows from (2). 

    (4) By (1), (2), and Corollary \ref{corollary3} (3),  
    $S$ is a far-flung Gorenstein numerical semigroup with $s(S)=\t(S)=t$. 
    We consider the case $t\ge9$. Then $\n(S)>\l(S)$ by (1), (3), and Proposition \ref{construction}. 
    Hence, since $S$ is far-flung Gorenstein, we have $\ress>\l(S)$.  
\end{proof}

\begin{remark}
Proposition \ref{example_Rohrbach} is related to the lower bound for $\ovn(t)$ given by Rohrbach.   
\end{remark}

\begin{theorem}\label{theorem_example}
    Let $t\ge 5$ be an integer. Then there exists a finite set $A \subset \N$ with $|A|=t$ such that $3(t-1)<n(A)$. 
    Moreover, there exists a far-flung Gorenstein numerical semigroup $S$ with $s(S)=\t(S)=t$ such that  $\ress>\l(S)$.
\end{theorem}
\begin{proof}
    By Proposition \ref{example_Rohrbach}, we only have to consider the case $t\in [5,8]$ as follows: 
    \begin{align*}
        & A_5:=  \{0,1,3,5,6\}, & n(A_5)=13 & \quad(t=5); \\
         & A_6:=  \{0,1,3,5,7,8\}, & n(A_6)=17 & \quad(t=6);  \\
         & A_7:=  \{0,1,3,5,6,8,12\}, & n(A_7)=19 & \quad(t=7);\\
         & A_8:=  \{0,1,3,5,6,8,12,16\}, & n(A_8)=23 & \quad(t=8).
    \end{align*}
    Note that $|A_t|=t$ for all $t\in [5,8]$. 
    We see that $\max A_t+2 \le n(A_t)$ and $3(t-1)<n(A_t)$ for all $t \in [5,8]$. 
    Hence, the same argument as in the proof of Proposition \ref{example_Rohrbach} (4) shows that 
    $S(n(A_t),A_t)$ is a far-flung Gorenstein numerical semigroup with $s(S(n(A_t),A_t))=\t(S(n(A_t),A_t))=t$ 
    such that $\res(S(n(A_t),A_t))>\l(S(n(A_t),A_t))$ for all $t \in [5,8]$.  
\end{proof}

\begin{corollary}
    Let $t\ge 5$ be an integer, and let $A \subset \N$ is an extremal finite set with $|A|=t$.  
    Then $3(t-1)<n(A)=\ovn(t)$. In particular, $S(n(A),A)$ is a far-flung Gorenstein numerical semigroup with $s(S(n(A),A))=\t(S(n(A),A))=t$ such that $\res(S(n(A),A))>\l(S(n(A),A))$.  
\end{corollary}

Regarding Question \ref{comparison}, the remaining case is $\t(S)=4$ (see \cite[Question 2.6]{zbMATH07574809}). 
We conclude this section by introducing a counterexample to this question. 

\begin{example}\label{type4}
    Let $S=\langle 13, 20, 21, 22, 23, 24, 27, 28, 31 \rangle$. 
    Then, we can check that $\Ks \setminus S=\{6,8,9,19,29,30,32\}$ and $S\setminus\trK=\{0,13,20,21,22,23,24,26\}$. 
    Thus, we see that 
    \[
    \ress=8>7=\l(S). 
    \]
    We can also check that $\rPF(S)=\PF(S)=\{29,30,32,38\}$, i.e., $s(S)=\t(S)=4$.
    We remark that $S$ is not far-flung Gorenstein by Lemma \ref{lemma1}. 
\end{example}

%%%%%%%%%%%%%%%%%%%%%%%%%%%%%%%%%%%%%%%%%%%%%%%%%%%%%%%%%%%%%%%

\section*{Acknowledgments}
I would like to thank Professor Satoru Fukasawa for valuable comments and suggestions. 
Several computations appearing in this paper were performed using the NumericalSgps package \cite{NumericalSgps1.4.0} in GAP \cite{GAP4.12.1}.


\begin{thebibliography}{10}

\bibitem{NumericalSgps1.4.0}
M.~Delgado, P.~A. Garcia-Sanchez, and J.~Morais.
\newblock {NumericalSgps}, a package for numerical semigroups, {V}ersion 1.4.0 dev, 2024.
\newblock GAP package.

\bibitem{zbMATH03993709}
R.~Fr{\"o}berg, C.~Gottlieb, and R.~H{\"a}ggkvist.
\newblock On numerical semigroups.
\newblock {\em Semigroup Forum}, 35:63--83, 1987.

\bibitem{grigorescu2025farflunggorensteinnumericalsemigroup}
T.~I. Grigorescu.
\newblock The far-flung {Gorenstein} numerical semigroup rings of type 4.
\newblock Preprint, {arXiv}:2512.18370 [math.AC], 2025.

\bibitem{zbMATH07395083}
J.~Herzog, T.~Hibi, and D.~I. Stamate.
\newblock Canonical trace ideal and residue for numerical semigroup rings.
\newblock {\em Semigroup Forum}, 103(2):550--566, 2021.

\bibitem{zbMATH07574809}
J.~Herzog and S.~Kumashiro.
\newblock Upper bound on the colength of the trace of the canonical module in dimension one.
\newblock {\em Arch. Math.}, 119(3):237--246, 2022.

\bibitem{zbMATH07648549}
J.~Herzog, S.~Kumashiro, and D.~I. Stamate.
\newblock The tiny trace ideals of the canonical modules in {Cohen}-{Macaulay} rings of dimension one.
\newblock {\em J. Algebra}, 619:626--642, 2023.

\bibitem{zbMATH07423434}
C.~Huneke, S.~Maitra, and V.~Mukundan.
\newblock Torsion in differentials and {Berger}'s conjecture.
\newblock {\em Res. Math. Sci.}, 8(4):15, 2021.
\newblock Id/No 60.

\bibitem{zbMATH06272466}
J.~Kohonen and J.~Corander.
\newblock Addition chains meet postage stamps: reducing the number of multiplications.
\newblock {\em J. Integer Seq.}, 17(3):article 14.3.4, 13, 2014.

\bibitem{arXiv:2306.17069}
S.~Maitra and V.~Mukundan.
\newblock Extremal behavior of reduced type of one dimensional rings.
\newblock Preprint, {arXiv}:2306.17069 [math.{AC}], 2023.

\bibitem{zbMATH03023950}
H.~Rohrbach.
\newblock Ein {Beitrag} zur additiven {Zahlentheorie}.
\newblock {\em Math. Z.}, 42:1--30, 1937.

\bibitem{zbMATH05623301}
J.~C. Rosales and P.~A. Garc{\'{\i}}a-S{\'a}nchez.
\newblock {\em Numerical semigroups.}, volume~20 of {\em Dev. Math.}
\newblock Dordrecht: Springer, 2009.

\bibitem{zbMATH06223909}
N.~J.~A. Sloane.
\newblock {On}-{Line} {Encyclopedia} of {Integer} {Sequences}, {https://oeis.org/A123509}.

\bibitem{sugawara2025classificationsymmetricnumericalsemigroups}
A.~Sugawara.
\newblock Classification of almost symmetric numerical semigroups with maximal reduced type.
\newblock Preprint, {arXiv}:2512.17957 [math.AC], 2025.

\bibitem{GAP4.12.1}
{The GAP~Group}.
\newblock {GAP} {\textendash} {G}roups, {A}lgorithms, and {P}rogramming, {V}ersion 4.12.1, 2022.

\end{thebibliography}
\end{document}